\documentclass{amsart}
\usepackage{latexsym}
\usepackage{amssymb}
\usepackage{amsmath}
\usepackage{enumerate}
\usepackage[dvips]{graphicx}
\usepackage[latin1]{inputenc}
\title{Pesin theory and equilibrium measures on the interval}
\author{Neil Dobbs}

\newcommand\omu{{\overline{\mu}}}

\newcommand\HD{\mathrm{HD}}

\newcommand\Supp{\mathrm{Supp}}

\newcommand\cM{\mathcal{M}}
\newcommand\cP{\mathcal{P}}

\newcommand\arr{\mathbb{R}}

\newcommand\R{\mathbb{R}}
\newcommand\Z{\mathbb{Z}}

\newtheorem{thm}{Theorem}
\newtheorem{dfn}[thm]{Definition}
\newtheorem{lem}[thm]{Lemma}
\newtheorem{prop}[thm]{Proposition}

\begin{document}
\date{\today}
\maketitle
\begin{abstract}
    We use Pesin theory to study possible equilibrium measures for piecewise monotone maps of the interval. The maps  may have unbounded derivative. 
\end{abstract}
\section{Introduction}

Our goal is to study possible equilibrium measures for rather general piecewise monotone maps of the interval, possibly with discontinuities. In \cite{Me:Cusp}, we developed Pesin theory for interval transformations with unbounded derivative and studied properties of measures absolutely continuous with respect to Lebesgue measure. Here we use the theory to investigate measures absolutely continuous with respect to some conformal measures, mirroring work we did in \cite{Me:Trans}. 

	In general, equilibrium measures are probability measures which encode dynamical information for large sets of points (those seen by the measure) with some property and they maximise (or minimise, depending on the definition) the \emph{free energy} with respect to the corresponding potentials. Measures of maximal entropy and absolutely continuous invariant measures are important examples of equilibrium measures.  Existence and uniqueness of equilibrium measures have long been of interest, see \cite{Hofbauer:IntrinsicErgodicityI, Hofbauer:IntrinsicErgodicityII, DKU:eq, Rychlik:BV, Walters:VP, Walters:beta, HK:Eq, BT:EquilibriumInterval, BruinTodd:Eq2, BRLSVS, LuzzattoTucker, LuzzattoViana:Ast00, DHL06}, for example. We show in this paper that equilibrium measures must often be of a certain form. In some cases this allows to show uniqueness (\cite{JuanLi:Hyp}). The main result is Theorem~\ref{thm:big}.

    The techniques and proofs rely heavily on those of \cite{Me:Cusp, Me:Trans}. Each of those articles extended work on Pesin theory of F.\ Ledrappier \cite{Ledrappier:AbsCnsInterval, Ledrappier:AbsCnsComplex}. Lemma~\ref{lem:9} is new, we provide some additional detail in Lemma~\ref{lem:qp}, and Proposition~\ref{lem:RuelleB} is different, but for the most part we refer to \cite{Me:Cusp, Me:Trans} for the proofs, so this paper remains brief.  

	We consider maps $f$ defined on a finite union of open intervals. This does not preclude gaps where $f$ is undefined and is thus quite general. Our results also apply to smooth multimodal maps, for example, as interesting measures tend not to live on the critical orbits, and removing critical points from the domain will give a cusp map. We allow unbounded potentials (from logarithm of the derivative) and also unbounded derivatives, providing an alternate approach to that of \cite{DKU:eq}. Using the natural extension and looking at measures with positive Lyapunov exponent lets us control distortion along almost all backward branches. 

    \begin{dfn}
    Let $I$ be a non-degenerate compact interval and let $I_1, \ldots, I_p$ be a finite collection of pairwise disjoint open subintervals of $I$. Following \cite{Me:Cusp}, a map $f : \bigcup_{j=1}^p I_j \to I$ is a \emph{weak piecewise monotone cusp map} (with constants $C, \epsilon >0$) if on each $I_j$,
    \begin{itemize}
        \item
        $f: I_j \to f(I_j)$ is a diffeomorphism;
        \item for all $x,x'$ such that $|Df_j(x)|, |Df_j(x')| \leq 2$,
            $$
                |Df_j(x) - Df_j(x')| \leq C |x-x'|^\epsilon;
            $$
        \item for all $x,x'$ such that $|Df_j(x)|, |Df_j(x')| \geq 2^{-1}$,
        $$
        \left|\frac{1}{Df_j(x)} - \frac{1}{Df_j(x')}\right| \leq C |x-x'|^\epsilon.
        $$
    \end{itemize}
    We say $f$ is a \emph{piecewise monotone cusp map} if, in addition, 
       on each $I_j$ and for each $x' \in \partial I_j$, 
       $
        \lim_{x \to x'} Df(x)$ equals 0 or $\pm \infty$. 
        \end{dfn}
        Thus the difference between being \emph{weak} or not is the derivative condition at the boundary of the domain of definition. Note that we assume that $f$ is defined on a finite collection of intervals. In \cite{Me:Cusp}, we started off allowing $f$ to be defined on a countable collection of intervals, but for some results subsequently assume that $f$ has at most a finite number of discontinuities (on top of being piecewise monotone). In this paper, the additional hypothesis appears in the definition.


We shall restrict our attention to Borel-measurable sets and Borel measures, without further mention. 

Let $\phi : I \to \arr$ be a H\"older continuous function and let $t \in \arr$. 
Consider the relation
\begin{equation}\label{eqn:conformality}
m(f(A)) = \int_A e^\phi |Df|^t dm.
\end{equation}
\begin{dfn} \label{dfn:phitconforme}
    Let $X', X \subset \arr$ and let $f: X' \to X$ be diffeomorphic on each connected component of $X'$. We say a
 measure $m$ is \emph{$(\phi,t)$-conformal} for $f$ if (\ref{eqn:conformality}) holds for
 every   set $A$ on which $f$ is injective. 
\end{dfn}
We emphasise that we do not require the conformal measure to be finite, though we will require it to be finite on some open subinterval for our main result. There are indications this could be useful in applications when it is sometimes hard to construct a well-behaved finite conformal measure. 
%

    We denote by $h_\mu$ the entropy of an invariant probability measure $\mu$. Given a \emph{potential} function $\psi : X \to \R$, we can define the \emph{pressure} 
    $$
    P(\psi, f) := \sup \left\lbrace h_\mu + \int_X \psi d\mu \right\rbrace
    $$
    where the supremum is taken over all invariant probability measures $\mu$. Any measure $\mu$ realising this supremum is called an \emph{equilibrium state} or \emph{equilibrium measure} (for the potential $\psi$). 
    An equilibrium state $\mu$ for $\psi$ is also an equilibrium state for $\psi +C$ for every constant $C$, and in particular for the potential $\psi_0 := \psi - P(\psi,f)$. Clearly $P(\psi_0, f) = 0$, so
    $$
    h_\mu = \int -\psi_0 d\mu,
    $$
    for the equilibrium state $\mu$. 
    Conversely, if $P(\psi_0, f) = 0$ and $h_\mu = \int -\psi_0 d\mu$, then $\mu$ is an equilibrium state. 
     Often one can show that there is a  $(\phi, t)$-conformal measure when $P(-\phi -t\log|Df|, f) = 0$. If that is the case, then a measure satisfying the equivalent conditions in the main theorem is an equilibrium state. 
     \begin{dfn}
        Let $f$ be a cusp map with a $(\phi,t)$-conformal measure $m$. Let $U$ be an open interval.  We call $G$ an \emph{expanding induced Markov map for $(f,m)$} if there is a countable collection of pairwise disjoint intervals $U_i \subset U$ such that:
        \begin{itemize}
            \item
                $m(U)>0$; 
            \item
                $m(U \setminus \bigcup_i U_i) = 0$;
            \item
                for each $i$, there exists $n_i$ such that $f^{n_i}_{|U_i} = G_{|U_i}  : U_i \to U$ is a diffeomorphism;
            \item
                there exist $C_0, \delta >0$ such that, for each $i$ and each $j \leq n_i$, 
                $$
                    |Df^j(x)| > C_0 e^{\delta j}$$
                for all $x \in f^{n_i - j}(U_i)$;
            \item 
                there exists $C_1 >0$ such that on each $U_i$, $G$ has distortion bounded by $C_1$ and  $|DG|>2$. 
        \end{itemize}
            If additionally $\sum_i n_i m(U_i) < \infty$, then we say that $G$ has \emph{integrable return time}. 
    \end{dfn}
    \begin{lem} \label{lem:normalise}
        Let $G$ be as per the definition and suppose $G$ has integrable return time.
            Let $\nu$ (cf. Lemma~\ref{lem:Folk}) be the unique absolutely continuous invariant probability measure for $G$. 
Set
    $$
    \mu' := \sum_i \sum_{j=0}^{n_i -1} f_*^j \nu.
    $$
             Then $\mu := \mu' /\mu'(I)$ is an ergodic, absolutely continuous $f$-invariant probability measure: 
    \end{lem}
        \begin{proof} Evident. \end{proof}
        \begin{dfn}
            The measure $\mu$ of Lemma~\ref{lem:normalise} is said to be \emph{generated by} $G$. 
    \end{dfn}

\begin{thm} \label{thm:big}
    Let $f : \bigcup_{j=1}^p I_j \to I$ be a weak piecewise monotone cusp map with a $(\phi, t)$-conformal measure $m$. 
    Let $\mu$ be an ergodic invariant probability measure with positive entropy and positive finite Lyapunov exponent $\chi_\mu$. 

	Suppose that $\Supp(\mu) \subset \Supp(m)$ and that there is an open interval $W$ with $\mu(W)>0$ and $m(W) < \infty$. 
    Then \begin{enumerate}[(a)]
    \item
        $h_\mu \leq t\chi_\mu + \int \phi d\mu$;
     \item if $P(-\phi -t\log|Df|, f) > 0$ then there is no equilibrium state with positive entropy and positive finite Lyapunov exponent. 
    \end{enumerate}
    The following conditions are equivalent. 
    \begin{enumerate}[(i)]
        \item \label{enum:ace1}$\mu \ll m$;
        \item \label{enum:ace4}$\mathrm{HD}(\mu) =  t + \int \phi d\mu/\chi_\mu$;
        \item \label{enum:ace2}$h_\mu = t\chi_\mu + \int \phi d\mu$;
        \item \label{enum:ace3}the density of $\mu$ with respect to $m$ is bounded from below by a positive constant on an open interval of positive measure;
        \item \label{enum:ace5} there is an expanding induced Markov map for $(f,m)$ with integrable return time which generates $\mu$.
    \end{enumerate}
    Should the equivalent conditions hold then 
     $P(-\phi -t\log|Df|, f) \geq 0$, with equality if and only if $\mu$ is an equilibrium state. 

    If one considers piecewise monotone cusp maps rather than \emph{weak} piecewise monotone cusp maps, then the references to positive entropy can be dropped. 
\end{thm}
Regarding proofs, (a) is shown in Proposition~\ref{lem:RuelleB}, and (a) implies (b). 
That (\ref{enum:ace1}) implies (\ref{enum:ace2}) is Lemma~\ref{lem:9}. 
That (\ref{enum:ace4}) if and only if (\ref{enum:ace2}) follows from Proposition~\ref{prop:dvl}. Lemmas~\ref{lem:den} and~\ref{lem:Mar} 
show (\ref{enum:ace2}) implies (\ref{enum:ace3}) and (\ref{enum:ace5}), while by definition, (\ref{enum:ace5}) implies (\ref{enum:ace1}). 

	Under some sort of transitivity assumptions, perhaps on the support of $m$, (\ref{enum:ace3}) will imply that $\mu$ is the \emph{unique} absolutely continuous (with respect to $m$) invariant probability measure. 
        
        While $f$ does not have critical points, conformal measures may be supported on points which get mapped outside the domain of definition of $f$. For example, given a unimodal map with a conformal measure supported on the critical point and its backward orbit, the measure will still be conformal for the corresponding cusp map (with the critical point removed from the domain of definition). 

	Supposing that $\Supp(m) \supset \Supp(\mu)$ is reasonable. For transitive maps, conformal measures usually have support equal to the entire space. For non-transitive maps, they will often still have support some large completely invariant set, for example the complement of a basin of attraction or some such. 
	
	Another way of looking at the first statement of the theorem is that conformal measures can only exist for certain combinations of $\phi$ and $t$, while existence of a conformal measure bounds the free energies of invariant probability measures.

\section{Proof}

We define the natural extension as per \cite{Ledrappier:AbsCnsComplex}. Let 
$$
Y := \{ y = (y_0y_1y_2\ldots) : f(y_{i+1}) = y_i \in I\}.
$$
Define $F^{-1} : Y \to Y$ by $F^{-1}((y_0y_1\ldots)) := (y_1y_2\ldots)$. Then $F^{-1}$ is invertible with inverse $F : F^{-1}(Y) \to Y$. The projection $\Pi : Y \to I$ is defined by $\Pi :y = (y_0y_1\ldots)  \mapsto y_0$. 
Then $\Pi \circ F = f \circ \Pi$. Given any measure $\mu \in \cM(f)$ there exists a unique $F$-invariant measure $\omu$ such that $\Pi_*\omu = \mu$. Moreover $\omu \in \cM(F)$ and $\omu \in \cM(F^{-1})$ (see \cite{Rohlin:Exact}). 

We call the triplet $(Y,F,\omu)$ the \emph{natural extension} of $(f,\mu)$ (it is also called the Rohlin extension or the canonical extension). 

Let us remark that invariant probability measures give no mass to the sets of points $x$ for which there is an $n > 0$ such that $f^n(x)$ is not defined, nor do they give mass to the set of $x$ for which there exists an $n>0$ and \emph{no solution} $x'$ to $f^n(x') = x$. Thus, $F^n(y)$ is defined for all $n \in \Z$ for $\omu$ almost every $y \in Y$.

We have the following unstable manifold theorem --- around almost every point in the natural extension one can pull back an interval along the corresponding branch as far as one wants with bounded distortion and exponential shrinking (and without meeting boundary points or discontinuities). 

\begin{thm}[Theorem 16 of \cite{Me:Cusp}]     \label{thm:NeilUnstableManif}
There exists a measurable function $\alpha$ on $Y$, $0 < \alpha < 1/2$ almost everywhere, such that for $\overline{\mu}$ almost every $y \in Y$ 
there exists a  set $V_y \subset Y$ with the following properties:
\begin{itemize}
\item $y \in V_y$ and $\Pi V_y =  B(\Pi y, \alpha(y))$;
\item for each $n >0$, $f^n : \Pi F^{-n}V_y \to \Pi V_y$ is a diffeomorphism (in particular it is onto);
\item for all $y' \in V_y$
$$
  \sum_{i=1}^\infty \left| \log|Df(\Pi F^{-i}y')| - \log |Df(\Pi F^{-i}y)|\right|  
< \log 2;
$$
\item for each $\eta >0$ there exists a measurable function $\rho$ on $Y$, $0< \rho(y) < \infty$ almost everywhere, such that 
$$
\rho(y)^{-1}e^{n(\chi - \eta)} < |Df^n(\Pi F^{-n}y)| < \rho(y)e^{n(\chi + \eta)}.
$$
\end{itemize}
In particular, $|\Pi F^{-n}V_y| \leq 2\rho(y)e^{-n(\chi-\eta)}$.
\end{thm}

With a non-trivial amount of work, one can then prove:

\begin{lem}[Lemma 25 of \cite{Me:Cusp}] \label{lem:exactness}
Given any interval $V$ of positive $\mu$-measure, there is a $j > 0$ such that $\mu\left(\bigcup_{k = 0}^j f^k(V)\right) =1$.
\end{lem}

    \begin{prop}[Proposition 30 of \cite{Me:Cusp}] \label{prop:dvl}
         $\HD(\mu) = h_\mu / \chi_\mu$.
    \end{prop}

\begin{dfn}[\cite{Graczyk:MetricAttractors}] An open interval $U$ is \emph{regularly returning} if $f^n(\partial U) \cap U = \emptyset$ for all $n > 0$. This is also called a \emph{nice interval} in the literature.
\end{dfn}
If $A$ is a connected component of $f^{-n}(U)$ and $B$ is a connected component of $f^{-m}(U)$ with $m \geq n$, it is easy to check that either $A\cap B = \emptyset$ or $B \subset A$, so  inverse images of regularly returning intervals are either nested or disjoint. Indeed, suppose $x \in \partial A \cap B$. Then $f^n(x) \in \partial U$ (since $f$ may be discontinuous, one uses that $x \in B$ to know that $f^n$ is defined on a neighbourhood of $x$), but $f^m(x) \in U$, contradiction. By Proposition~28 of \cite{Me:Cusp}, almost every point is contained inside arbitrarily small regularly returning intervals. 

Recall that $W$ is an open interval with $\mu(W) >0$ and $m(W)<\infty$. 
\begin{prop}
    There exist a regularly returning interval $U \subset W$ and  a constant $K> 0$ and a set $A \subset Y$ with the following properties:
    \begin{itemize}
        \item
            $\omu(A) >0$;
        \item 
            for $y \in A$, there is a $y' \in A$ such that $y \in V_{y'}$, $\rho(y') < K$, $\Pi V_{y'} \supset U$ and $U_y := V_{y'} \cap \Pi^{-1}U \subset A$.
    \end{itemize}
\end{prop}
	\begin{proof} This follows from Theorem~\ref{thm:NeilUnstableManif} and the ubiquity of regularly returning intervals. 
	\end{proof}
		
    Thus we can fix a regularly returning interval $U$, a corresponding set $A$ and for each $y \in A$ there is a set $U_y\subset A$ with $\Pi U_y = U$. Moreoever, $f^n$ maps $\Pi F^{-n}U_y$ diffeomorphically and with distortion bounded by $2$ onto $U$. 
    
    \begin{prop} \label{prop:nicenice}
    There exists a measurable partition $\xi$ of $Y$ such that, if $\xi(y)$ denotes the element of $\xi$ containing $y$ and $e(y)$ the first entry time of $y$ to $A$, then 
    $$
    F^{e(y)}\xi(y) = U_{F^{e(y)}y}.$$
    For any $y, y'$ in $Y$, $\Pi \xi(y)$ and $\Pi \xi(y')$ are either nested or disjoint, and $\xi(y)$ and $\xi(y')$ are either nested or disjoint. The entropy of $\mu$ is given by $h_\mu = H(F^{-1}\xi/\xi)$. 
    \end{prop}
    \begin{proof} The claim about the entropy is shown in Proposition 30 of \cite{Me:Trans}. 
    \end{proof}
    Note that we have chosen a definition of $\xi$ corresponding to that of \cite{Me:Trans} rather than \cite{Me:Cusp} as things are slightly simpler this way. In particular, we get uniform bounds  in the following lemma with this definition of the partition.

Let $\psi := \phi + t\log|Df|$. Let 
$$
    (S_{-n})\psi(y) := \sum_{i=1}^n \psi \circ \Pi \circ F^{-i}(y).$$
    \begin{lem} \label{lem:snpsi}
        There are uniform (independent of $n, y$ and $y' \in \xi(y)$) upper and lower bounds on $(S_{-n}\psi)(y') - (S_{-n}\psi)(y)$ and $\lim_{n \to \infty} \left((S_{-n}\psi)(y') - (S_{-n}\psi)(y)\right)$ exists. 
        \end{lem}
        \begin{proof}
            See the proof of Lemma 31 of \cite{Me:Trans}. Exponential decrease of preimages gives a bound on the H\"older part, and Theorem~\ref{thm:NeilUnstableManif} takes care of the derivative part. 
            \end{proof}

    Define $\Phi(y,\cdot) : \xi(y) \to \arr$ by
        $$
        \Phi(y,y') := \lim_{n\to \infty} e^{(S_{-n}\psi)(y') - (S_{-n}\psi)(y)},
        $$
	so $\Phi(y,\cdot)$ is uniformly bounded away from zero and infinity, from Lemma~\ref{lem:snpsi}. 

        \begin{lem} \label{lem:9}
            Suppose $\mu$ is absolutely continuous with respect to $m$. Then $h_\mu = t \chi_\mu + \int \phi d\mu$. 
            \end{lem}
            \begin{proof}
                Let $\cP$ denote the generating partition $\{I_1, \ldots, I_p\} \vee \{I\setminus U, U\}$ (see Proposition~29 of \cite{Me:Cusp}). Denote by $\cP_n(x)$ the element of $\bigvee_{i=0}^n f^{-i}\cP$. For $\mu$-almost every $x$ in $I$, there is a sequence $n_j = n_j(x)$ tending to infinity for which $\cP_{n_j}(x)$ is the projection of an element of $F^{-n}\xi$.  It follows that, using Lemma~\ref{lem:snpsi} and conformality, 
	$$
	-\frac{1}{n} \log m\left(\cP_{n_j}(x) \right) $$
	converges to $t \chi_\mu + \int \phi d\mu$. Meanwhile, the Shannon McMillan-Breiman Theorem says that 
	$$
	-\frac{1}{n} \log \mu\left(\cP_{n_j}(x) \right) $$
	converges to $h_\mu$ almost everywhere.  Set $\gamma := h_\mu - t \chi_\mu - \int \phi d\mu$. Thus for a set $X$ with $\mu(X) = 1$ and every $x\in X$, there are arbitrarily large $n$ and open sets $P_n(x) \subset W$ with 
	$$
	\left|
	-\frac{1}{n} \log \mu\left(P_n(x) \right) 
	+\frac{1}{n} \log m\left(P_n(x) \right)  -\gamma \right| < |\gamma/2|.
$$
	Now $0 < \mu(X \cap W), m(X\cap W) < \infty$. 
If $\gamma >0$ we deduce, covering $X$ by such sets, for all large $N$,
	$$\mu(W) = \mu(X \cap W)  \leq e^{-N\gamma/2}m(W)< m(W)  < \infty.$$ 
	Letting $N$ tend to infinity we deduce $\mu(W) = 0$, a contradiction. 
	On the other hand, if $\gamma < 0$ then similarly 
	$$\mu(X \cap W)  \geq e^{-N\gamma/2}m(X\cap W)  \geq m(X\cap W).$$ 
	Letting $N \to \infty$ we derive a contradiction with absolute continuity, as $m(X\cap W) >0$. 
            \end{proof}
	Note we actually proved something extra in deriving a contradiction from $\gamma >0$, namely that $h_\mu \leq t\chi_\mu +\int \phi d\mu$, which we shall reprove shortly. Also, in the proof, we could take our cover of $X$ to be pairwise disjoint because of the regularly returning property. Even without the regularly returning property, the Besicovitch Covering Theorem would have done the job.

            Let $p(y,\cdot)$ denote the Rohlin decomposition of $\omu$ with respect to $\xi$, so, writing $\xi_{-n}$ for $F^{-n}\xi$ for clarity,
    \begin{equation} \label{eqn:ph}
     n h_\mu = -\int_Y \log p(y, \xi_{-n}(y)) d\omu,
    \end{equation}
    as per (11) of \cite{Me:Cusp}. We note that $p(y, \xi_{-n}(y)) > 0$ almost everywhere. 
    Let 
    $$ q(y, dz) := \frac{\Phi(y,z) dm_{\xi(y)}(z)}{\int_{\xi(y)} \Phi(y,y') dm_{\xi(y)}(y')},
    $$
    where $m_{\xi(y)}$ is the pullback by $\Pi_{|\xi(y)}$ of the conformal measure $m$ restricted to $\Pi \xi(y)$. Because $\Supp(m) \supset \Supp(\mu)$, and elements of $\xi_{-n}$ project onto open intervals, $q(y, \xi_{-n}(y))>0$ $\omu$-almost everywhere. However, $q(\cdot,\xi_{-n}(\cdot))$ may be positive on sets (of zero $\omu$-measure) where $p(\cdot,\xi_{-n}(\cdot))$ is not. The function $q$ is our best guess as to how $p$ would look, were $\mu$ absolutely continuous, informed by the change of variables formula and the notion that on elements of $\xi_{-n}$ for large $n$, the densities should be almost constant. 
    \begin{lem}
        \begin{equation} \label{eqn:qchi}
    -\int \log q(y, \xi_{-n}(y)) d\omu = n \left(t\chi_\mu +  \int \phi d\mu\right).
    \end{equation}
    \end{lem}
    \begin{proof}
        As shown in the proof of Proposition 32 of \cite{Me:Trans}.
    \end{proof}
    Comparing $q$ and $p$ will, using equations (\ref{eqn:ph}) and (\ref{eqn:qchi}), allow us to relate $h_\mu$ and $\chi_\mu$.
    Define $\overline{\nu}$ on measurable subsets of $Y$ by 
    $$
    \overline{\nu}(B) = \int_Y q(y,B) d\omu(y).
	$$
	Denote by $Y_n$ the quotient space $Y|\xi_{-n}$ and by $\overline{\nu_n}$ and $\omu_n$ the corresponding push-forwards of $\overline{\nu}$ and $\omu$ under the quotient map. 
    For each point $V_n \in Y_n$, we define $q_n(V_n) := q(y, V)$ and $p_n(V_n):= p(y,V)$ for any $y \in V \subset Y$, where $V$ is the element of $\xi_{-n}$ which projects to $V_n$. One can check that $p_n >0$ almost everywhere with respect to $\omu_n$. 
    \begin{lem} \label{lem:qp}
        Almost everywhere with respect to $\omu_n$,
     $$
     q_n/p_n = d\overline{\nu}_n/d\omu_n.
     $$
 \end{lem}
 \begin{proof}
    Let $\kappa, c, \varepsilon >0$. Let $H_n \subset Y_n$  be a set on which $p_n > \kappa > 0$ and $|q_n - c p_n | < \varepsilon$. Let $H$ be the corresponding subset of $Y$. Then $|q(y, H) - c p(y,H)| < \varepsilon/\kappa$ 
    almost everywhere (noting that $q, p \leq 1$). Let $H^* := \{y : \xi(y) \cap H \ne \emptyset \}$. On $Y \setminus H^*$, $q(y, H) = p(y,H) =0$, while on $H^*$, $p(y, H^*) = 1 \leq p(y, H)/\kappa$. Integrating the latter gives $\omu(H^*) \leq \omu(H)/\kappa$. Writing $\theta:= \int_Y \left( q(y,H) -cp(y,H) \right) d\omu$, we derive that $|\theta| \leq \varepsilon \omu(H) /\kappa^2$. 
     Thus
     \begin{eqnarray*}
    \int_{H_n} \frac{d\overline{\nu}_n}{d\omu_n} d\omu_n &=&  
        \overline{\nu_n}(H_n) = \int_Y q(y, H) d\omu = \theta + \int_Y cp(y, H) d\omu \\
        &=& \theta +  \int_H c d\omu \\
        &=& \theta + \int_{H_n} q_n/p_n d\omu_n + \int_{H_n} (c - q_n/p_n) d\omu_n.
        \end{eqnarray*}
        Therefore 
        $$
    \int_{H_n} \frac{d\overline{\nu}_n}{d\omu_n} d\omu_n =  \int_{H_n} q_n/p_n d\omu_n + \varepsilon_*
        $$
    for some $|\varepsilon_*| \leq \omu_n(H_n) 2\varepsilon/\kappa^2$. We deduce that for any set $J_n \subset Y_n$ with $p_n > \kappa>0$ on $J_n$, 
        $$
        \int_{J_n} \left| \frac{d\overline{\nu}_n}{d \omu_n} - \frac{q_n}{p_n} \right| d\omu_n \leq 2\varepsilon/\kappa^2. 
        $$
    Letting $\varepsilon$ and then $\kappa$ go to zero,  the lemma follows. 
    \end{proof}

    \begin{prop} \label{lem:RuelleB}
        $h_\mu  \leq t \chi_\mu + \int \phi d\mu.$ If equality holds, then $q=p$ and $\mu$ is absolutely continuous with respect to $m$. 
        \end{prop}
    \begin{proof}
        By equations (\ref{eqn:ph}) and (\ref{eqn:qchi}), 
        $$
        n\left(h_\mu - t\chi_\mu - \int \phi d\mu \right) = \int_{Y_n} \log \frac{q_n}{p_n} d\omu_n \leq \log \int_{Y_n} \frac{q_n}{p_n} d\omu_n,
        $$
        the latter by concavity of logarithm. 
        But by Lemma~\ref{lem:qp}, the latter expression is bounded above by $\log \overline{\nu_n}(Y) = 0$ (it may be negative if there is a set $S$ with $\omu_n(S) = 0$ and $\overline{\nu}(S)>0$). Thus,
        $$
            h_\mu \leq t\chi_\mu + \int \phi d\mu.
        $$
        Equality can only hold if $q_n = p_n$ almost everywhere. If this holds for all $n$, then $q=p$ almost everywhere, so $\mu$ is absolutely continuous.
    \end{proof}
        \begin{lem} \label{lem:den}
            Suppose $h_\mu = t\chi_\mu +\int \phi d\mu$. Then the density of $\mu$ with respect to $m$ is bounded away from zero on the inteval $U$. 
            \end{lem}
        \begin{proof}
             We have $\omu(B)  > 0$. For $y \in A$, the Rohlin decomposition $p(y,\cdot)=q(y, \cdot)$ has 
            density on $\xi(y)$ uniformly bounded away from zero, from the definition of $q$. Let $\omu_A$ denote the restriction of $\omu$   to $A$. 
            The measure $\Pi_* \omu_A$ therefore has density bounded away from zero. For any $C \subset U$, $\mu(C) \geq \Pi_* \omu_A(C)$ and so $\mu_{|U}$ also has density bounded away from zero. 
            \end{proof}

        \begin{lem} \label{lem:Folk}
            Let $G$ be an expanding induced Markov map for $(f,m)$ with range $W$. Then there is an ergodic absolutely continuous $G$-invariant probability measure $\nu$ with density uniformly bounded away from zero and infinity $m$-almost everywhere on $W$. 
            \end{lem}
        \begin{proof}
        See Lemma~35 of \cite{Me:Trans} for the proof, which is a little more involved than the standard Folklore Lemma. 
            \end{proof}
                
        \begin{lem} \label{lem:Mar}
            Suppose $h_\mu = t\chi_\mu +\int \phi d\mu$. There exists an expanding induced Markov map for $(f,m)$ with integrable return time which generates $\mu$. 
            \end{lem}
        \begin{proof}
        See Proposition 34 of \cite{Me:Trans} for the proof. 
            \end{proof}

\section{Acknowledgments}
The author is very grateful to Huaibin Li, Juan Rivera-Letelier, Kay Schwieger and Mike Todd for helpful conversations.     
  
\bibliography{references} 
\bibliographystyle{plain}
\end{document}